\documentclass{amsart}

\usepackage{preamble}

\title{On invariant subrings of Orlik--Solomon and Varchenko--Gel'fand algebras in type A}
\author{Trevor Karn}
\date{\today}

\begin{document}

\maketitle

\begin{abstract}
    We provide simple presentations in terms of generators and relations for the invariant subring of both the Orlik--Solomon algebra and Varchenko--Gel'fand ring of the type $A_n$ reflection arrangement acted upon by the type $A_{n-1}$ reflection group. 
    This may be interpreted as a presentation for the cohomology of the ``mixed configuration space" of $n$ red points and one blue point.
    We provide increasingly refined descriptions of the invariant ring starting with the total dimension and ending with the simple presentation in terms of generators and relations.
\end{abstract}

\tableofcontents

\section{Introduction}

The primary goal of this paper is to provide a simple presentation for a certain ring in terms of its generators and relations.
For any ring $R$ carrying the action of a group $G$, the $G$-invariant subring of $R$ is
\begin{align*}
R^G := \{x \in R : g \cdot x = x \; \forall g \in G\}.
\end{align*}
The two fundamental questions of invariant theory are (i) determine a generating set for $R^G$ and (ii) determine relations among those generators.
Our main results, \Cref{thm:VGnPresentation,thm:OSnPresentation}, answer these two fundamental questions for specific choices of $R$ and $G$ that arise in the study of hyperplane arrangements.

Classical invariant theory focuses on the case where $R$ is a polynomial ring.
The most well studied groups which act on the polynomial ring are reductive algebraic groups and reflection groups.
Theorems of Chevalley and Shephard--Todd (see \cite[Chapter~3]{humphreys}) give a particularly elegant answer to these questions when $G$ is a reflection group, such as the symmetric group $\symm_n$.
Whereas classical invariant theory starts with a group acting on a polynomial ring, we will instead start with a group acting on the Orlik--Solomon algebra $\OS(\A)$ and the Varchenko--Gel'fand ring $\VG(\A)$.
For the particular arrangement $\A_n$ that is the focus of this paper, $\OS(\A_n)$ and $\VG(\A_n)$ are defined below.

When a hyperplane arrangement $\A$ has the symmetry of a  group $G$, both $\VG(\A)$ and $\OS(\A)$ ``see" this symmetry and thus carry $G$ representations.
The prototypical example of a hyperplane arrangement with the symmetry of a finite group is the rank-$n$ braid arrangement 
\[\A_{n} := \{h_{ij} : 1 \leq i < j \leq n+1\},\] where $h_{ij} \subseteq \C^{n+1}$ is the hyperplane defined by $z_i = z_j$.
This is the reflection arrangement associated to the symmetric group $\symm_{n+1}$, which is the type $A_n$ Coxeter group.
The symmetric group $\symm_{n+1}$ acts on $\A_n$ by permuting the coordinates of $\C^{n+1}$.
This action restricts to $\symm_n$ by fixing the last coordinate and only permuting the first $n$ coordinates.

The two main results of this paper are presentations, in terms of generators and relations, for the rings $\OS(\A_n)^{\symm_n}$ and $\VG(\A_n)^{\symm_n}$.
We write $\VG_n := \VG(\A_n)$ and $\OS_n := \OS(\A_n)$.
We will also write the invariants $\invVG := \VG_n^{\symm_n}$ and $\invOS := \OS_n^{\symm_n}.$
Our first theorem answers the two fundamental questions of invariant theory for $\invVG$.

\begin{theorem} \label{thm:VGnPresentation}
    For $n \geq 2$, $\invVG \cong \Q[z]/\langle z^{n+1}\rangle$.
\end{theorem}

$\invOS$ is slightly more subtle. 
Consider the free associative algebra $\Q\langle A \rangle$ on an alphabet $A$.
Endow the alphabet $A$ with a degree function $|\cdot |: A \to \Z$.
The \textit{free graded-commutative algebra} $\Q\{A\}$ is the quotient of $\Q\langle A \rangle$ by the two-sided ideal $\langle a b - (-1)^{|a||b|} ba : a,b \in A\rangle$.

\begin{example}
    Let $A = \{\alpha, \mu, \gamma\}$ with degree function $|\alpha| = |\mu| = 1$ and $|\gamma| = 2$.
    Then in $\Q \{A\}$, the following identities hold:
    \begin{itemize}
        \item $\alpha^2 = \mu^2 = 0$,
        \item $\gamma^d \neq 0$ for all $d \geq 0$,
        \item $\alpha \mu = -\mu\alpha$,
        \item $\alpha \gamma = \gamma \alpha$ and $\mu \gamma = \gamma \mu$.
    \end{itemize}
    $\Q\{\alpha, \mu, \gamma\}$ can be described as the tensor product of an exterior algebra with a univariate polynomial ring, that is $\Q\{\alpha, \mu, \gamma\}\cong \wedge\{\alpha, \mu\} \,\otimes\, \Q[\gamma]$.
    The isomorphism is one of \textit{graded} algebras when the polynomial generator has degree $2$.
    The monomials of the form $\alpha^\epsilon \mu^{\delta} \gamma^d$ are a $\Q$-basis for $\Q\{A\}$ when $\epsilon,\delta \in \{0,1\}$ and $d \in \Z_{\geq 0}$.
\end{example}

\begin{definition}\label{def:Ideal}
    Let $p = \lfloor \frac{n+1}{2} \rfloor$ and define the following two-sided ideal of $\Q\{\alpha, \mu, \gamma\}$:
    \[I_n = 
    \begin{cases}
        \langle \gamma^p \rangle &\text{if $n$ is even}\\
        \langle \gamma^p\;,\; \alpha \gamma^{p-1} - (p-1) \mu \gamma^{p-1}\rangle &\text{if $n$ is odd}
    \end{cases} \]
\end{definition}

Our second main theorem, whose proof makes up the bulk of this paper, answers the two fundamental questions of invariant theory for $\invOS$.

\begin{theorem} \label{thm:OSnPresentation}
    For $n \geq 2$,
    $\invOS \cong \Q\{\alpha,\mu,\gamma\}/I_n$.
\end{theorem}

Graded-commutative algebras often arise naturally in topology, especially as cohomology rings. 
As we will explain below, $\invOS$ and $\invVG$ have interpretations as cohomology rings.
Thus, one should not be surprised that $\invOS$ may be presented as a quotient of a free graded-commutative algebra.
Orlik and Solomon \cite{orlik-solomon} showed that the topology of the complement of a hyperplane arrangement and the combinatorics of hyperplane intersections are deeply related.

In the case of $M_2(\A_n) := \C^{n+1} - \A_n$, the complement of the arrangement consists of points in $\C^{n+1}$ where no two coordinates are the same. 
Regarding each complex coordinate $z_i$ as a point in the complex plane $\C \cong \R^2$, this is the same as saying that $M_2(\A_n)$ is the \textit{ordered configuration space of points in the plane}.
The fundamental group of $M_2(\A_n)$ is well known to be the braid group, giving the braid arrangement its name.
A related topological space is the \textit{unordered configuration space of indistinguishable points in the plane}.
Cohen \cite[Part~3]{cohen-lada-may} described the homology (and the corresponding induced $\symm_{n+1}$ action) of the unordered configuration space as a Gerstenhaber algebra.
The unordered configuration space is formed by identifying points in the in the same orbit under the natural $\symm_{n+1}$ action on the ordered configuration space $M_2(\A_n)$; in other words, the unordered configuration space is the orbit space $M_2(\A_n)/\symm_{n+1}$.

There is also a \textit{mixed configuration space} that lives between the ordered and unordered configuration spaces.
A further configuration space to consider is that of $n+1$ points in the plane where the first $n$ points are indistinguishable from each other and the last point is distinguishable from any of the first $n$ points. 
To simplify this we will call the first $n$ points ``red points" and the last point a ``blue point." Thus, there is a mixed configuration space consisting of configurations of $n$ red points and one blue point.
As an orbit space, it is $M_2(\A_n)/(\symm_{n} \times \symm_1)$.
The study of orbit spaces naturally connects algebraic topology to invariant theory.
The connection comes through the isomorphism $H^\bullet(X/G;\Q) \cong H^\bullet(X;\Q)^G$ when $G$ acts freely on $X$ (see for example \cite[Section~3.G]{hatcher} for further details).
Thus, one may regard \Cref{thm:OSnPresentation} as a presentation for the cohomology of the configuration space of $n$ red points and $1$ blue point.

In the past decade, attention has turned to look at the relationship between subspace arrangement complements in Euclidean spaces and combinatorics.
This has led to study of the Varchenko--Gel'fand ring $\VG(\A)$.
When a subspace arrangement in $\R^{d n}$ comes as the thickening $\A \otimes \R^d$ of a real hyperplane arrangement $\A \subseteq \R^n$, Moseley showed that (up to a linear shift of grading) the cohomology of a subspace arrangement complement is either an Orlik--Solomon algebra or a Varchenko--Gel'fand ring.
Moseley's work inspires the following philosophy: given a nice result about either the Orlik--Solomon algebra or the Varchenko--Gel'fand ring, one hopes there is a similar nice result for the other.
This philosophy was the motivation for investigating \Cref{thm:VGnPresentation}.

As mentioned above, the complement of the braid arrangement is the unordered configuration space of points in the plane.
But it is also just one example of a reflection group acting freely on the complement of reflecting hyperplanes.
More generally, reflection groups $W$ provide a rich collection of examples where a group acts freely on the complement of a hyperplane arrangement.
Douglass, Pfeiffer, and R\"ohrle \cite{douglass-pfeiffer-roehrle-20,douglass-pfeiffer-roehrle-25} have systematically studied the corresponding $W$-representations of $\OS(\A_W)$.
With an eye towards invariant theory, one is first led to wonder about the structure of the invariant rings $\OS(\A_W)^W$ and $\VG(\A_W)^W$.
One of the important classes of subgroups of reflection groups are the parabolic subgroups $P \leq W$.
In the context of our main results, one wonders if there is a general story to be told about the invariant rings $\OS(\A_W)^P$ and $\VG(\A_W)^P$.
Further investigation of this story is left for future work.

The remainder of this paper is organized as follows.
\Cref{sec:Background} describes some of the background which we will draw upon to prove \Cref{thm:VGnPresentation,thm:OSnPresentation}.
\Cref{sec:VGInvariants} is dedicated to the proof of \Cref{thm:VGnPresentation}.
\Cref{sec:HilbOS} focuses on computing the Hilbert--Poincar\'e polynomial of $\invOS$, which is the content of \Cref{prop:OSinvariantHilb}. 
\Cref{sec:LinearBasis} describes two linear bases for $\invOS$ and focuses on showing the linear independence of the bases; spanning will follow from \Cref{prop:OSinvariantHilb}.
Lastly, \Cref{sec:RingStructure} will draw upon the linear basis to prove \Cref{thm:OSnPresentation}.

\section{Background}\label{sec:Background}

There is a famous linear basis for $\OS(\A)$ of any arrangement known as the NBC basis.
\begin{definition}
    Fix an order on the hyperplanes $h_1 < h_2 < \cdots < h_k$.
    A \textit{circuit} is a set of hyperplanes $C = \{h_i\}$ such that $\codim \cap_{h \in C} h < |C|$ and $\codim \cap_{h \in C-h'} h = |C-h'|$ for any $h' \in C$.
    In other words, a circuit is a minimally dependent set of hyperplanes.
    A set $S$ is called a \textit{broken circuit set} if there is a hyperplane $h \in \A$ such that $ \{h\} \cup S$ is a circuit and $h < h_i$ for all $h_i \in S$.
    A \textit{no broken circuit (NBC) set} is a set which does not contain any broken circuit as a subset.
    Generators of $\OS(\A)$ (respectively $\VG(\A)$) are indexed by hyperplanes, and so the image of a monomial in $\OS(\A)$ (respectively $\VG(\A)$) corresponds to a set of hyperplanes. 
    If the corresponding index set of hyperplanes is an NBC set, we say that the monomial is an \textit{NBC monomial}.
\end{definition}

\begin{remark}
    We will often abuse notation and refer to monomials in $\OS(\A)$ or $\VG(\A)$. 
    This is taken to mean the images of monomials in $\OS(\A)$ and $\VG(\A)$, which are defined as quotients. 
\end{remark}

For any fixed order on the hyperplanes of $\A$, the NBC monomials form a basis for $\OS(\A)$ and $\VG(\A)$. Indeed as graded vector spaces we have $\OS(\A) \cong \VG(\A)$, but the isomorphism is emphatically not one of rings or of representations.
In the special case of the braid arrangement $\A_n$, Barcelo \cite[{Theorem~5.1}]{barcelo} described an elegant combinatorial model for the NBC sets.

\begin{definition}\label{def:HandsFingers}
    Define a hand $H_i$ to be the labeled complete bipartite graph $K_{1,i-1}$ where the vertex in the first connected component is labeled by $i$ and the rest are labeled by $1,2,\ldots,i-1$.
    The vertex $i$ is called the \textit{wrist} and the edges $(i,j)$ are called the \textit{fingers}.
    A subset of the disjoint union $\sqcup_{i=2}^n H_i$ is called an $n$-\textit{handful}.
\end{definition}

Each hyperplane $h_{ij} \in \A_n$ corresponds to both the edge $(i,j)$ in the hand $H_i$ and a generator $e_{ij}$ of $\OS_n$.
Thus, $(n+1)$-handfuls are in bijection with monomials in $\OS_n$ and $\VG_n$.
Order the hyperplanes $h_{ij} \in \A_n$ by the lexicographic order on the subscripts $h_{ij}$.
Then a monomial in $\OS(\A_n)$ or $\VG(\A_n)$ is an NBC monomial if and only if the corresponding handful consists of at most one finger from each hand.
The fact that there are $i-1$ ways to choose at most one finger from the hand $H_i$ provides a quick combinatorial proof of the well-known fact that $\dim \OS(\A_n) = \dim \VG(\A_n) = (n+1)!$.

\begin{example}
    Consider $H_2 \sqcup H_3 \sqcup H_4$, drawn below:
    \[\hand{2}{1} \hand{3}{1,2} \hand{4}{1,2,3}\]
    Thus, the monomials $e_{12}e_{23}e_{24}$ and $x_{12}x_{23}x_{24}$ (in $\OS(\A_3)$ and $\VG(\A_3)$ respectively) are NBC monomials and part of the NBC bases.
    Conversely, $e_{14}e_{34}$ is not an NBC monomial.
    Indeed, $\{h_{13} < h_{14} < h_{34}\}$ is a circuit.
\end{example}

The last piece of structure we wish to describe for the Orlik--Solomon algebra is that of an acyclic chain complex.
What we say below holds in greater generality than we state. For the full story see for example \cite[Lemma~3.13]{orlik-terao}.
Define the map 
\[ \hat\partial: \wedge \{e_{ij}\} \to \wedge \{e_{ij}\}\]
on degree-$d$ exterior monomials by linearly extending the alternating sum of all degree-$d-1$ monomials dividing the original one.
Explicitly,
\begin{equation}\label{def:DefnOfPartial}
\hat \partial(e_{i_1j_1}e_{i_2j_2}\cdots e_{i_dj_d})
=
\sum_{k=1}^d (-1)^{k+1} e_{i_1j_1}e_{i_2j_2}\cdots \hat{e}_{i_k j_k}\cdots e_{i_dj_d}
\end{equation}
where $\hat{e}_{i_k j_k}$ means to omit only the factor of ${e}_{i_k j_k}$ in the product.
The map $\hat\partial$ descends to a map 
$\partial:\OS_n \to \OS_n$.
Moreover,
\begin{lemma}[{\cite[{Lemma~3.13(c)}]{orlik-terao}}]
    \label{lem:AcyclicPartial}
$\partial$ determines the acyclic chain complex
\[ \begin{tikzcd}
(\OS_n)_0 \cong \Q & 
\arrow[l,"\partial"'] (\OS_n)_1 & 
\arrow[l,"\partial"'] (\OS_n)_2 & 
\arrow[l,"\partial"'] \cdots &
\arrow[l,"\partial"'] (\OS_n)_n &
\arrow[l,"\partial"'] 0.
\end{tikzcd}
\]
\end{lemma}
In particular this means that $\partial:(\OS_n)_n \to (\OS_n)_{n-1}$ is injective.
Furthermore $\partial$ is a differential in the sense that $\partial(ab) = \partial(a) b + (-1)^{|a|} a \partial(b)$ where $|a|$ is the degree of $a$.
In particular this implies the power rule
\begin{equation}\label{eqn:PartialPowerRule}
    \partial(a^d) = d a^{d-1} \partial(a).
\end{equation}

We now turn our attention to the $\symm_n$-invariants.
In order to prove \Cref{thm:VGnPresentation,thm:OSnPresentation}, we will provide increasingly refined descriptions of $\invOS$ and $\invVG$.
The coarsest description of $\invOS$ and $\invVG$ is their dimension as vector spaces.
The main tool we will use to understand these dimensions is \textit{Frobenius reciprocity}.
Recall that we are focused on studying $\symm_n$-representations obtained by restricting $\symm_{n+1}$-representations.
Accordingly, we will provide the shortest (rather than the most general) path to understand $\dim \invOS$ and $\dim \invVG$.
Experts will be well aware of how much more generally Frobenius reciprocity may be stated. 

The irreducible representations of $\symm_n$ are in bijection with partitions $\lambda \vdash n$.
The Frobenius characteristic map $\Frob$ associates to an irreducible representation $V^\lambda$ the Schur function
\[s_\lambda = \sum_{\SSYT(\lambda)} \x^{\wt T}\]
where $\SSYT(\lambda)$ is the set of all semistandard tableaux of shape $\lambda$ and $\wt T$ is the monomial $\prod_i x_i^{m_i}$ where $m_i$ is the multiplicity of the integer $i$ in a given filling of the tableau.
This association is an isometry of inner product spaces between the representation ring, for which irreducible representations form an orthonormal basis, and the ring of symmetric functions, for which $\{s_\lambda\}_{\lambda\vdash n}$ is an orthonormal basis. 
For full details, see \cite{sagan}, \cite[Chapter~7]{ec2}, and \cite{macdonald}, in that order.
We will write the inner product as $\langle -, -\rangle$.

For a representation $V$ obtained by restricting a $\symm_{n+1}$ representation to $\symm_n$, the multiplicity of $V^\lambda$ is $\langle \Frob V, s_\lambda h_1 \rangle.$
Let $V^{\symm_n}$ denote the trivial isotypic component of the representation $V$.
The trivial representation corresponds to the symmetric function $h_n = s_n$, and so to find the multiplicity of the trivial representation inside of $V$, we wish to compute
\begin{equation*}
\langle \Frob V, h_n h_1 \rangle = \dim V^{\symm_n}.
\end{equation*}

In order to use Frobenius reciprocity as described above to compute the dimension of $V^{\symm_n}$, we need to know $\Frob V$.
Felder and Veselov \cite[p.~103]{felder-veselov} determined a description of the representation of a Coxeter group on the Orlik--Solomon algebra of the corresponding hyperplane arrangement. 
As a special case of their more general result,\footnote{This special case was first understood by Lehrer and Solomon in \cite{lehrer-solomon}.} one obtains isomorphism of $\symm_{n+1}$-representations
\[\OS_n \cong \left( \trivrep \uparrow_{\langle (12) \rangle}^{\symm_{n+1}} \right)^{\oplus 2}.\]
In terms of symmetric functions, this means 
\begin{equation}\Frob \OS_n = 2h_2h_1^{n-1}.\end{equation}

As a consequence of work of Moseley, \cite[Proposition~6.1]{moseley}, the $\symm_{n+1}$ action on $\VG_n$ is the regular representation. Thus,
\begin{equation}
    \Frob \VG_n = h_1^{n+1}.
\end{equation}

The Pieri rule is a way to compute the symmetric function product $s_\lambda h_k$ in terms of the Schur basis.
It states that $s_\lambda h_k$ is the sum over all Schur functions $s_\mu$ where the Ferrers diagram of the skew shape $\mu/\lambda$ has $k$ cells in total, with no two cells in the same column.
In the special case when $\lambda = (\lambda_1)$ is a single row, then $s_\lambda = h_{\lambda_1}$ and so the sum is over $s_\mu$ where $\mu$ has at most two rows and the first row has length at least $\max(\lambda_1, k)$ and the second row has length at most $\min(\lambda_1,k)$.

\begin{example}\label{ex:TotalDims}
We can use Frobenius reciprocity to compute the following vector space dimensions:
\begin{enumerate}
    \item $\dim \OS_n^{\symm_{n+1}} = \langle 2h_2h_1^{n-1}, h_{n+1}\rangle = 2$,
    \item $\dim \invOS = \langle 2h_2h_1^{n-1}, h_nh_1\rangle = 2n$,
    \item $\dim \VG_n^{\symm_{n+1}} = \langle h_1^{n+1}, h_{n+1}\rangle = 1$,
    \item $\dim \invVG = \langle h_1^{n+1},h_nh_1\rangle = n+1$.
\end{enumerate}
These follow from iterating the Pieri rule, together with multilinearity of $\langle -,\, - \rangle$.
\end{example}

While the calculus of symmetricfunctionology allows us to get our hands dirty and do computations of vector space dimensions, there is also a topological interpretation of $\invOS$ and $\invVG$.
Work of Arnol'd \cite{arnold} and Moseley \cite{moseley} shows that $\OS_n$ and $\VG_n$ are the cohomology rings of the ordered configuration spaces of $n+1$ points in $\R^d$ when $d\geq 2$ is even and odd respectively, up to a uniform shift of grading.
This may be rephrased as the complement of the subspace arrangements $\{x_i = x_j\}$.
Define the space 
\[M_d(\A_n) := \R^{d(n+1)} - \A_n \otimes \R^d.\]

When a group $G$ acts on a topological space $X$ so that the projection $\pi: X \to X/G$ is a covering map, for example when $G$ acts freely on $X$, then there is a relationship between their rational cohomology rings.
Explicitly, the action of $G$ on $X$ induces an action of $G$ on $H^\bullet(X)$.
This induces the following isomorphism:
\[H^\bullet(X/G;\Q) \cong H^\bullet(X;\Q)^G.\]
See for example \cite[\S~3.G]{hatcher} for further detail.
Since the fixed points of the $\symm_{n+1}$ action on $(\R^d)^{n+1}$ correspond to the subspaces where the $i$th and $j$th coordinates (themselves $d$-tuples of real numbers) are equal, the action is free and thus $\symm_{n+1}$, and indeed any subgroup, acts freely on $M_d(\A_n)$. 
So for any $G \leq \symm_{n+1}$ and $d \geq 2$ we obtain the ring isomorphism
\[H^\bullet(M_d(\A_n)/G;\Q) \cong H^\bullet(M_d(\A_n);\Q)^G \cong \begin{cases}
\VG_n^G & \text{$d$ odd}\\
\OS_n^G & \text{$d$ even}\\
\end{cases}.\]
In particular, taking $G = \symm_n$, we see that
\begin{equation}
    H^\bullet(M_d(\A_n)/\symm_n;\Q) \cong
    \begin{cases}
    \invVG & \text{$d$ odd}\\
    \invOS & \text{$d$ even}\\
    \end{cases}.
\end{equation}
We can interpret $M_d(\A_n)/\symm_n$ as the configuration space of $n+1$ points in $\R^d$, the first $n$ of which are unordered and the last of which is distinguishable from the others.
More simply, $M_d(\A_n)/\symm_n$ is the \textit{mixed configuration space} of $n$ red points and $1$ blue point in $\R^d$.
Hoang \cite{hoang-b,hoang-char} studied mixed configuration spaces with a view towards number theory.

\section{Varchenko--Gel'fand invariants}\label{sec:VGInvariants}

Cohen \cite[Chapter~III]{cohen-lada-may} gave a presentation for $\VG_n$, stated purely in terms of cohomology.
In our notation, Cohen's presentation is:
\begin{equation}\label{eqn:CohenVGPresentation}
\VG_n \cong \Q[x_{ij}]_{1 \leq i \neq j \leq n+1}/\langle x_{ij}^2, x_{ij} + x_{ji}, x_{ij}x_{jk} + x_{jk}x_{ki} + x_{ki}x_{ij}\rangle.
\end{equation}
Of note, there is only one family of degree-$1$ relations, namely $x_{ij} + x_{ji}$.
The $\symm_{n+1}$ action on $\VG_n$ permutes the subscripts of the $x_{ij}$ variables.\footnote{In fact there is also an interesting $\symm_{n+2}$ action discovered by Whitehouse \cite{whitehouse} and studied more recently by Early and Reiner \cite{early-reiner}.}
The $\symm_{n+1}$ restricts to a $\symm_n$ action by considering permutations fixing $1$.
The orbit sum $z=\sum_{i=2}^{n+1} x_{1,i}$ is fixed under this $\symm_n$ action.

The presentation also shows that $(ij)\cdot x_{ij} = x_{ji} = -x_{ij}$.
A standard calculation then shows that the Frobenius characteristic of the degree-$1$ component of $\VG_n$ is $e_2h_1^{n-1}$.
The pairing $\langle e_2 h_1^{n-1}, h_n h_1 \rangle = 1$ shows that $\dim \VG_n^{\symm_{n}} = 1.$
Having found a nonzero degree-$1$ invariant element, namely $z$, one concludes that the degree-$1$ component of $\VG_n^{\symm_n}$ is the $\Q$-span of $\{z\}$.

\begin{lemma}\label{lem:PowersOfVGGenerator}
    Let $z = \sum_{i=2}^{n+1} x_{1,i} \in \VG_n$.
    Then $z^d \neq 0$ for $0\leq d\leq n$.
\end{lemma}

\begin{proof}
    Recall from \Cref{def:HandsFingers} the combinatorial model for no-broken-circuit (NBC) sets of $\A_n$ by handfuls with no two fingers on the same hand.
    The sum $z$ is the multivariate generating function in the variables $x_{1,f}$ for ways to choose a single finger $(1,f)$.
    Since $x_{ij}^2=0$, we interpret $z^d$ as the multivariate generating function for ways to choose exactly $d$ distinct fingers $\{(1,f_1), (1,f_2),\ldots (1,f_d)\}=F$. 
    The finger $(1,f_i)$ belongs to the hand $H_{f_i}$, and so every finger in $F$ belongs to a different hand.
    Then $z^d$ is a positive sum of NBC monomials for $0 \leq d \leq n$ and so $z^d \neq 0$.
\end{proof}

We are now ready to prove \Cref{thm:VGnPresentation}.

\begin{proof}[Proof (of \Cref{thm:VGnPresentation})]

\Cref{ex:TotalDims} tells us that $\dim \VG_n^{\symm_n} = n+1$.
\Cref{lem:PowersOfVGGenerator} tells us that $B = \{1,z,z^2,\ldots,z^n\}$ are all nonzero.
Furthermore, since each $z^d$ lives in a different degree, $B$ is evidently linearly independent.
Since $|B|=n+1$, it is a basis for $\dim \VG_n^{\symm_n}$.
To see that $z^{n+1} = 0$ in $\VG_n$, use the fact that $\VG_n = \oplus_{d=0}^{n} \VG_{n,d}$ is only nontrivial in degree $d \leq n$, together with the observation that $z^{n+1}$ has degree $n+1$.
\end{proof}

\section{The Hilbert-Poincar\'e polynomial of Orlik--Solomon invariants}\label{sec:HilbOS}

The story for the Orlik--Solomon algebra is more subtle.
There is a presentation for $\OS_n$ that is similar to \Cref{eqn:CohenVGPresentation}, originally due to Arnol'd \cite{arnold}:
\begin{equation}\label{eqn:OSPresentation}
    \OS_n \cong \wedge \{e_{ij}\}_{1 \leq i \neq j \leq n+1} / 
    \langle e_{ij} - e_{ji},\, e_{ik}e_{jk} - e_{ij}e_{jk} + e_{ij}e_{ik} \rangle.
\end{equation}
The relation $e_{ij}^2 = 0$ in the exterior algebra mirrors the $x_{ij}^2 = 0$ relation in $\VG_n$.

In their landmark paper \cite{orlik-solomon}, Orlik and Solomon were able to determine that the graded $G$ representation structure of $\OS(\A)$ for any arrangement $\A$ where $G$ acts by permuting hyperplanes depends only on the lattice of flats $\L(\A)$.

\begin{lemma}[{\cite[{Cor. 5.6}]{orlik-solomon}}]
    Let $WH_{\bullet}(\L(\A))$ denote the \textit{Whitney homology} of the \textit{lattice of flats} $L(\A)$ of the arrangement $\A$. 
    Then \[\OS(\A) \cong WH_{\bullet}(\L(\A))\] as graded representations.
\end{lemma}

Using the combinatorics of $\L(\A)$ to compute the representation structure allows for a better understanding of the structure of $\OS(\A)$.
As one example of this, Sundaram \cite{sundaram} proved the following equivariant recursion for a finite poset $P$ that is Cohen--Macaulay.
Recall this means that $P$ is ranked and that the reduced homology of the order complex $\Delta(x,y)$ of the open interval $(x,y)$ vanishes below the top degree $\rk(y)-\rk(x)-2$.

\begin{lemma}{\cite[Prop.~1.9]{sundaram}}
    Let $P$ be a Cohen--Macaulay poset of rank $n$ and let
    $P(r) = \{x \in P : \rk x \leq r\}.$
    Then, for $1\leq r \leq n - 1$ there is an isomorphism of representations
    \[\Tilde{H}(P(r)) \oplus \Tilde{H}(P(r-1)) \cong WH_r(P),\]
    when taking the convention that $\Tilde{H}(P(0))=WH_0(P)$ and $\Tilde{H}P(r))=0$ for $r < 0$ and $r > n-1$.
\end{lemma}

Without precisely defining what it means for a poset to be Cohen--Macaulay, we will say that the set partition lattice $\Pi_{n+1}$ is a prototypical example of a Cohen--Macaulay poset.
This is a particularly important example because of the well known fact that $\Pi_{n+1}$ is isomorphic to the lattice of flats $\L(\A_n)$ of the braid arrangement.
Let $\OS_{n,d}$ and $\invOS[n,d]$ denote the degree-$d$ graded components of $\OS_n$ and $\invOS$ respectively.
Sundaram's result implies the following corollary.

\begin{corollary}\label{cor:OSnddirectsum}
    Let $V_{n+1}(d) = \Tilde{H}_d(\Pi_{n+1})$.
    Then, as graded representations,
    \[\OS_{n,d} \cong WH_{d}(\Pi_{n+1}) \cong V_{n+1}(d) \oplus V_{n+1}(d-1).\]
\end{corollary}

Sundaram \cite{sundaram} was also able to determine the multiplicity of some irreducible $\symm_n$ representations inside of the homology $\Tilde{H}_\bullet(\Pi_n(d))$.

\begin{lemma}{\cite[Cor.~2.3]{sundaram}}\label{lem:SundaramCor2.3}
    Let $n \geq 3$ and $0 \leq r \leq n-2$ and let $V_n(r) = \Tilde{H}(\Pi_n(r))$.
    Then
    \begin{enumerate}
        \item $\displaystyle \langle V_n(r), V^{n}\rangle = 0$
        \item $\displaystyle \langle V_n(r), V^{n-1,1} \rangle = 1$
    \end{enumerate}
    where $V^\lambda$ is the irreducible representation of $\symm_n$ such that $\Frob(V^\lambda) = s_\lambda$.
\end{lemma}

In \Cref{ex:TotalDims}, we computed $\dim \invOS = 2n.$
We refine that computation into the Hilbert series of $\invOS$.

\begin{proposition}\label{prop:OSinvariantHilb}
    For $n \geq 2$,
    $\Hilb(\invOS;t)=(1+t)(1+t+t^2 + \cdots + t^{n-1}).$
\end{proposition}

\begin{proof}
    
    Recall that $\OS_{n,d}$ is naturally thought of as an $\symm_{n+1}$ representation.
    By Frobenius reciprocity,
    \begin{align*}
        \dim {\invOS[n,d]}
        &= 
        \langle \OS_{n,d} \downarrow_{\symm_n},
        \trivrep_{\symm_n}\rangle_{\symm_n} \\
        &= 
        \langle \OS_{n,d}, \trivrep_{\symm_n}\uparrow^{\symm_{n+1}}\rangle_{\symm_{n+1}}\\
        &=\langle \OS_{n,d}, 
        V^{n+1} \oplus V^{n,1}\rangle_{\symm_{n+1}}.
    \end{align*}
    By \Cref{cor:OSnddirectsum,lem:SundaramCor2.3}
    \begin{align*}
        \langle \OS_{n,d}, 
        V^{n+1} \oplus V^{n,1}\rangle_{\symm_{n+1}}
        &=
        \langle 
            V_{n+1}(d) \oplus V_{n+1}(d-1),
            V^{n+1} \oplus V^{n,1}\rangle_{\symm_{n+1}}
        \\
        &= \begin{cases}
            2 & 1 \leq d \leq n-1\\
            1 & d=0,n\\
            0 & \text{otherwise}
        \end{cases}\qedhere
    \end{align*}
\end{proof}

\section{A linear basis for $\OS_n^{\symm_n}$}\label{sec:LinearBasis}

\Cref{prop:OSinvariantHilb} described the dimensions of $\OS_n^{\symm_n}$ as a graded vector space.
This section is dedicated to proving a linear basis for $\OS_n^{\symm_n}$.
The basis will consist of homogenous elements, and thus will describe a basis for each graded component of $\OS_n^{\symm_n}$.

\begin{definition}\label{def:BasisElements}
    Let $\hat a, \hat m, \hat c, \hat g \in \wedge \{e_{ij}\}_{1 \leq i < j \leq n}$ be defined by
    \begin{align*}
        \hat a &:= \sum_{1 \leq i < j \leq n} e_{ij}\\
        \hat m &:= \sum_{1 \leq i  \leq n} e_{i,n+1}\\
        \hat c &:= \sum_{1 \leq i < j \leq n} \left( e_{ij}e_{i,n+1} + e_{ij}e_{j,n+1}\right)\\
        \hat g &:= am - pc.
    \end{align*}
    Let $a,m,c,g$ be the images in $\OS_n$ of $\hat a, \hat m, \hat c, \hat g$.
\end{definition}

\begin{remark}
    Note that all of the monomials in the definition of $\hat a, \hat m, \hat c, \hat g$ are NBC monomials. 
    Thus, they have the same expression whether written in the exterior algebra or in $\OS_n$.
    The presence of the $\hat {\cdot}$ is to remind the reader when we want to regard a sum as an element of the exterior algebra rather than the Orlik--Solomon algebra.
    The $\hat{\cdot}$ symbol is chosen because it looks like the exterior wedge symbol $\wedge$.
\end{remark}
\begin{remark}
    By construction, $a,m,c$ are all $\symm_n$ orbit sums, and thus $a,m,c,g \in \OS_n^{\symm_n}$.
\end{remark}

\begin{proposition}\label{prop:LinearBasisC}
    Let $n \geq 2$. Let
        \[
    B = \begin{cases}
    \{a^\epsilon\, m^\delta\, c^d : \epsilon, \delta \in \{0,1\} \text{ and } 0 \leq d < p \} & n~\text{even} \\
    \{a^\epsilon\, m^\delta\, c^d : \epsilon, \delta \in \{0,1\} \text{ and } 0 \leq d < p -1 \} \cup \{c^{p-1}, ac^{p-1} + mc^{p-1}\} & n~\text{odd}.
    \end{cases}
    \]
    Let $B_k := \{b \in B : \deg b = k\}$.
    Then
    $B = \sqcup_{k=0}^n B_k$ is a graded linear basis for $\OS_n^{\symm_n}$.
\end{proposition}

\begin{remark}\label{remark:LinearBasisG}
    The element $g$ defined in \Cref{def:BasisElements} does not appear in $B$.
    Indeed, by construction $\{am, c, g\}$ is linearly dependent.
    We define both $c$ and $g$ because $c$ will be easier to work with while considering the vector space structure of $\OS_n^{\symm_n}$ while $g$ will be useful in \Cref{sec:RingStructure} below to describe the ring theoretic structure.
    Let
    \[
    B' = \begin{cases}
    \{a^\epsilon\, m^\delta\, g^d : \epsilon, \delta \in \{0,1\} \text{ and } 0 \leq d < p \} & n~\text{even} \\
    \{a^\epsilon\, m^\delta\, g^d : \epsilon, \delta \in \{0,1\} \text{ and } 0 \leq d < p -1 \} \cup \{g^{p-1}, ag^{p-1} + mg^{p-1}\} & n~\text{odd},
    \end{cases}
    \]
    and $B'_k = \{b \in B' : \deg b = k\}$.
    It is straightforward to see that \Cref{prop:LinearBasisC} implies that $B' = \sqcup_{k=0}^n B'_k$ is also a graded basis for $\OS_n^{\symm_n}$.
\end{remark}

First we show linear independence in small degrees.

\begin{lemma}\label{lemma:Deg1-2LI}
    \begin{enumerate}
        \item\label{item:Deg1LI} For $n \geq 2$, $\{a,m\} = B_1$ is linearly independent.
        \item\label{item:Deg2LI} For $n \geq 3$, $\{am, c\} = B_2$ is linearly independent.
    \end{enumerate}
\end{lemma}
\begin{proof}
    To see \Cref{item:Deg1LI}, observe that $a,m$ are nonempty sums over disjoint degree-one elements of the exterior algebra $\wedge \{e_{ij}\}$ and the Orlik--Solomon relations in $\OS_n$ are generated in degree two.

    When $n = 3$, note that $am = c$.
    When $n \geq 4$ (the setting of \Cref{item:Deg2LI}) first note that $am, c \neq 0$.
    Indeed
    \[am = \sum_{1 \leq i < j \leq n}\sum_{1 \leq k  \leq n} e_{ij}e_{k,n+1}.\]
    Each monomial $e_{ij}e_{k,n+1}$ is an NBC monomial, so the expression for $am$ above is a positive sum of basis elements of $\OS_{n,2}$, hence $am\neq 0$.
    Similarly, since 
    \[c = \sum_{1 \leq i < j \leq n} \left( e_{ij}e_{i,n+1} + e_{ij}e_{j,n+1}\right)\]
    is a positive sum of NBC monomials, $c \neq 0$.

    To see that $\{am, c\}$ is linearly independent, observe that if $n \geq 4$,
    \begin{align*}
        0 = k_1 am + k_2 c 
        &= k_1  \sum_{1 \leq i < j \leq n}\sum_{1 \leq k  \leq n} e_{ij}e_{k,n+1} \\
        &\quad + k_2 \sum_{1 \leq i < j \leq n} \left( e_{ij}e_{i,n+1} + e_{ij}e_{j,n+1}\right) \\
        &= (k_1 + k_2) \sum_{1 \leq i < j \leq n}\sum_{\substack{k \in \{i,j\}}} e_{ij}e_{k,n+1}\\
        &\quad + k_1 \sum_{1 \leq i < j \leq n}\sum_{\substack{k \not\in \{i,j\}}} e_{ij}e_{k,n+1}.
    \end{align*}
    Since the NBC monomials form a basis $\OS_{n,2}$, we see that $k_1 = 0$ and $k_1 + k_2 = 0$, hence $k_2 = 0$.
\end{proof}

The technical tool which we will use to proof \Cref{prop:LinearBasisC} is the well-known fact that the differential map $\partial$ defined in \Cref{def:DefnOfPartial} satisfies the Leibniz rule: $\partial(xy) = \partial(x) y + (-1)^{|x|} x \partial(y)$.
The following identities will allow us to use this fact to prove the linear independence of $B$.

\begin{lemma}\label{lem:PartialIdentities}

    \begin{enumerate}
        \item $\partial(a) = \binom{n}{2}$
        \item $\partial(m) = n$
        \item $\partial(am) = -na + \binom{n}{2}m$
        \item $\partial(c^d) = -2dac^{d-1} + d(n-1)mc^{d-1}$
        \item $\partial(amc^d) = -nac^d + \binom{n}{2}mc^d$
        \item $\partial(g) = \begin{cases}
            0 & n~\text{even}\\
            a - (p-1)m& n~\text{odd}
        \end{cases}$
    \end{enumerate}
\end{lemma}
\begin{proof}
    Since $\partial(e_{ij})=1$, we see that $\partial(a)$ (respectively $\partial(m)$) is an integer counting the number of monomials in $a$ (respectively $m$) which is $\binom{n}{2}$ (respectively $n$).
    Then
    \begin{align*}
        \partial(am)
        &= \partial(a)m - a \partial(m)\\
        &= \binom{n}{2}m - na
    \end{align*}

    The Leibniz rule implies the power rule that $\partial(c^d) = dc^{d-1}\partial(c)$.
    Thus, we compute
    \begin{align*}
        \partial(c) &= \partial\left( \sum_{1 \leq i < j \leq n} e_{ij}e_{i,n+1} + e_{ij}e_{j,n+1}\right) \\
        &=\sum_{1 \leq i < j \leq n} e_{i,n+1} - e_{ij} + e_{j,n+1} - e_{ij}\\
        &=(n-1) m - 2a.
    \end{align*}
    So
    \[\partial(c^d) = dc^{d-1}\left((n-1)m - 2a\right).\]
    Next,
    \begin{align*}
        \partial(amc^d)
        &= \partial(am)c^d + am \partial(c^d)\\
        &= \left( -na + \binom{n}{2}m \right) c^d + am \left(-2dac^{d-1} + d(n-1)mc^{d-1}\right)\\
        &= \left( -na + \binom{n}{2} m \right) c^d.
    \end{align*}
    Lastly,
    \begin{align*}
        \partial(g)
        &= \partial(am - pc)\\
        &= (-na + \binom{n}{2}m) - p \partial(c^1)\\
        &= (-na + \binom{n}{2}m) - p \left(-2a + (n-1)m \right)\\
        &= \left(-n +2p\right) a + \left(\binom{n}{2} -(n-1)p\right) m.
    \end{align*}
    When $n$ is even, $p = \lfloor\frac{n+1}{2} \rfloor = \lfloor\frac{n}{2}\rfloor = \frac{n}{2}$.
    When $n$ is odd, $p = \frac{n+1}{2}$.
    Thus,
    \begin{align*}
        \partial(g)
        &=
        \begin{cases}
           \left (-n+ 2 \frac{n}{2} \right)a + \left( \frac{n(n-1)}{2} - (n-1)\frac{n}{2}\right)m & n~\text{even}\\
           \left (-n+ 2 \frac{n+1}{2} \right)a + \left( \frac{n(n-1)}{2} - (n-1)\frac{n+1}{2}\right)m  & n~\text{odd}
        \end{cases}\\
        &=
        \begin{cases}
            0 & n~\text{even}\\
            a - (p-1)m& n~\text{odd}
        \end{cases}
        \qedhere
    \end{align*}
\end{proof}

The following proposition relates linear independence in consecutive degrees. 

\begin{lemma} \label{lemma:ConsecutiveDegreeLI}
Suppose that $n \neq  3$.
    \begin{enumerate}
        \item  If $\{ac^d,\, mc^d\}$ are linearly independent 
        and $\dim \OS_{n,2d+2}^{\symm_n} \geq 2$, then $\{amc^d, c^{d+1}\}$ is linearly independent.
        \item \label{prop:EvenImpliesOddLI} If $\{amc^{d-1},\, c^d\}$ are linearly independent and $\dim \OS_{n,2d+1}^{\symm_n} \geq 2$, then $\{ac^d, mc^d\}$ is linearly independent.
    \end{enumerate}
\end{lemma}

\begin{proof}
Suppose that $\{ac^d, mc^d\}$ are linearly independent.
Let $k_1, k_2 \in \Q$ so that $k_1 amc^d + k_2 c^{d+1} = 0$.
Then $\partial(k_1 amc^d + k_2 c^{d+1}) = 0$.
On the other hand, by applying the identities of \Cref{lem:PartialIdentities} we can express the same element in terms of $a,m,c$: 
\begin{align*}
    \partial(k_1 amc^d + k_2 c^{d+1})
    &= \left(k_1(m-a) + k_2((n-1)m - 2a) \right) c^d \\
    &\quad +  \left(k_1 am + k_2c \right) d c^{d-1}((n-1)m - 2a)\\
    &=  \left ( -k_1-2k_2-2dk_2\right ) ac^d + \left (k_1 +(n-1)k_2+ (n-1)dk_2 \right ) mc^d.
\end{align*}
Because $\{ac^d, m_c^d\}$ is linearly independent, we obtain the system of equations that
\[
\begin{cases}
    0 = -k_1 - 2k_2 - 2d k_2 \\
    0 = k_1 + (n-1)k_2 + (n-1)dk_2
\end{cases}
\]
which has the unique solution $k_1=k_2=0$ when $n \neq 3$.
The assumption that the ambient dimension is at least two is used implicitly; if $amc^d$ and $c^{d+1}$ were in a one-dimensional space, they must be linearly dependent.
The proof of \Cref{prop:EvenImpliesOddLI} is similar.
\end{proof}

The previous result relates linear independence in consecutive degrees.
After addressing the base case in the following lemma, we will be able to bootstrap the linear independence of the putative basis $B$.

The combination of \Cref{prop:OSinvariantHilb,lemma:ConsecutiveDegreeLI,lemma:Deg1-2LI} proves that the elements of $B$ of degree strictly less than $n$ form a basis for $\oplus_{d=0}^{n-1}\OS_{n,d}^{\symm_n}$.
To finish the poof of \Cref{prop:LinearBasisC} that $B$ is a basis, we need to show that the unique degree-$n$ element of $B$ is nonzero.
The nonvanishing in degree-$n$ follows from the following stronger statement.

\begin{lemma}\label{lemma:LIimpliesNonzero}
    \begin{enumerate}
        \item If $\{ac^d, mc^d\}$ is linearly independent, then $amc^d \neq 0$
        
        \item\label{item:LIImpliesNonzero} If $\{amc^{d-1}, c^{d}\}$ is linearly independent, then $ac^{d}  + mc^{d} \neq 0$.
    \end{enumerate}
\end{lemma}
\begin{proof}
    Suppose for the sake of contradiction that $\{ac^d, mc^d\}$ is linearly independent and $amc^d = 0$.
    If this were true then by \Cref{lem:PartialIdentities}, we would have
    \[0 = \partial(0) = \partial(amc^d) = (m-a)c^d,\]
    contradicting the linear independence of $ac^d$ and $mc^d$.
    The proof of \Cref{item:LIImpliesNonzero} is similar.
\end{proof}

We are finally ready to prove \Cref{prop:LinearBasisC}, that $B$ is a linear basis for $\OS_n^{\symm_n}$.

\begin{proof}[Proof~(of \Cref{prop:LinearBasisC})]
    In degree-$0$, $\OS_{n,0} = \OS_{n,0}^{\symm_n} \cong \Q$, so $a^0m^0c^0 = 1$ is the basis element of degree-$0$.
    By \Cref{prop:OSinvariantHilb}, we see that $\dim \OS_{n,d}^{\symm_n} = 2$ for $1 \leq d \leq n-1$.
    By \Cref{lemma:Deg1-2LI} we have produced two linearly independent elements in each of $\OS_{n,1}^{\symm_n}$ and $\OS_{n,2}^{\symm_n}$.
    By induction, \Cref{lemma:ConsecutiveDegreeLI} implies that $B_{<n}=\sqcup_{k=0}^{n-1} B_{k}$ is a basis for $\oplus_{d=0}^{n-1} \OS_{n,d}^{\symm_n}$.
    By \Cref{prop:OSinvariantHilb}, we see that $\dim \OS_{n,n}^{\symm_n} = 1$.
    \Cref{lemma:LIimpliesNonzero} shows that the unique degree-$n$ element of $B$ is nonzero, hence a basis for $\OS_{n,n}^{\symm_n}$.
\end{proof}

\section{Generators and relations of invariants in the Orlik--Solomon algebra}\label{sec:RingStructure}

This section is dedicated to proving \Cref{thm:OSnPresentation}.
Recall the two sided ideal $I_n$ defined in \Cref{def:Ideal}.
We do this by constructing an explicit isomorphism between $\Q\{\alpha, \mu, \gamma \}/I_n$ and $\OS_n^{\symm_n}$

\begin{definition}\label{def:IsomorphismMap}
    Let $a,m,g \in \OS_n$ be defined as in \Cref{def:BasisElements}.
    Let $p := \lfloor \frac{n+1}{2}\rfloor$ as in \Cref{thm:OSnPresentation}.
    Define the algebra homomorphism $\varphi: \Q\{\alpha,\mu,\gamma \} \to \OS_n$ by
    \[\varphi(\alpha) = a, \quad \varphi(\mu) = m, \quad \text{and} \quad \varphi(\gamma) = g.\]
\end{definition}

\begin{proposition}\label{prop:ExplicitIsomOS}
    $\varphi$ descends to an isomorphism $\Q\{\alpha,\mu,\gamma \}/I_n \to \OS_n^{\symm_n}$
\end{proposition}
In \Cref{lemma:ImVarPhi} we will show that $\OS_n^{\symm_n} = \img \varphi$.
In \Cref{lemma:KerVarPhi} we will show that $I_n = \ker \varphi$.
\Cref{prop:ExplicitIsomOS} then follows from a straightforward application of Noether's First Isomorphism Theorem that $\img \varphi \cong \Q\{\alpha, \mu, \gamma\}/\ker \varphi.$

\begin{lemma}\label{lemma:ImVarPhi}
    $\varphi$ surjects onto $\OS_n^{\symm_n}$.
\end{lemma}

\begin{proof}
    We see that $\varphi(\alpha^\epsilon \mu^\delta \gamma^d) = a^\epsilon m^\delta g^d$.
    Comparing this to \Cref{remark:LinearBasisG} we see that every basis element of $\OS_n^{\symm_n}$ is in the image of $\varphi$, hence $\varphi$ is surjective.
\end{proof}

We now begin to analyze the kernel of $\varphi$.
Recall that $\OS_{n,n+1} = 0$ for all $n$.
Recall from \Cref{lem:AcyclicPartial} that $\OS_n$ forms an acyclic chain complex with respect to $\partial$.
In particular, this tells us that 
$\partial: \OS_{n,n} \to \OS_{n,n-1}$ is injective.

\begin{lemma} \label{lemma:PowerofgIs0}
    Let $g$ and $p$ be as in \Cref{def:IsomorphismMap}. Then $g^p = 0$.
\end{lemma}

\begin{proof}
    When $n$ is odd, $2p = n+1$. Thus $g^p \in \OS_{n,n+1}=0$.

    When $n \geq 2$ is even, $g^p \in \OS_{n,n}$. Suppose for the sake of contradiction that $g^p \neq 0$.
    Then since $\dim \OS_{n,n} = 1$ by \Cref{prop:OSinvariantHilb}, we could write $kg^p = amg^{p-1}$ for some $k \neq 0$.
    If this were true, then by injectivity of $\partial: \OS_{n,n}\to\OS_{n,n-1}$ we would see that
    $\partial(kg^p - amg^{p-1}) = 0$.
    Thus we compute
    \begin{align*}
        \partial(kg^p - amg^{p-1})
        &= k p g^{p-1} \partial(g) - \left(\partial(am)g^{p-1} + am \partial(g^{p-1}) \right)\\
        &= 0 - \partial(am)g^{p-1} - am(p-1)g^{p-1} \partial(g)& (\text{by \Cref{lem:PartialIdentities}})\\
        &= nag^{p-1} + \binom{n}{2}mg^{p-1}. 
    \end{align*}
    But we know from \Cref{remark:LinearBasisG} that $ag^{p-1}$ and $mg^{p-1}$ are linearly independent, and so to have $\partial(kg^p - amg^{p-1})=0$ it must be that $n = \binom{n}{2} = 0$ contradicting that $n \geq 2$.
\end{proof}

\begin{lemma} \label{lemma:LinearDependenceInTopDegree}
    Let $g$ and $p$ be as in \Cref{def:IsomorphismMap}. Furthermore, let $n$ be odd.
    Then
    \begin{equation}\label{eqn:LinearDependenceInTopDegree}
        a g^{p-1} = (p-1) m g^{p-1}
    \end{equation}
\end{lemma}

\begin{proof}
    We compute
    \begin{align}
        \partial(ag^{p-1})
        &= 
        \partial(a)g^{p-1} - a \partial(g^{p-1}) \nonumber \\
        &=
        \binom{n}{2}g^{p-1} - (p-1)ag^{p-2}\partial(g)\nonumber \\
        &=
        (p-1)ng^{p-1} + (p-1)^2amg^{p-2} \nonumber \\
        &=
        (p-1)\left( \partial(m) g^{p-1} + m (p-1) g^{p-2}\partial(g) \right)\nonumber \\
        &= (p-1)\partial(mg^{p-1}). \label{eqn:PartialIdentityPenultimateDegree}
    \end{align}
    Since $\partial: \OS_{n,n} \to \OS_{n,n-1}$ is injective, \Cref{eqn:PartialIdentityPenultimateDegree} implies \Cref{eqn:LinearDependenceInTopDegree}.
\end{proof}

\begin{lemma}\label{lemma:KerVarPhi}
    \[\ker \varphi =
    \begin{cases} \langle \gamma^p \rangle &\text{if $n$ is even}\\
    \langle \gamma^p,\; \alpha \gamma^{p-1} - (p-1)\mu \gamma^{p-1}\rangle &\text{if $n$ is odd.}\\
    \end{cases}\]
\end{lemma}
\begin{proof}
    \Cref{lemma:PowerofgIs0} shows that $\langle \gamma^p\rangle \subseteq \ker \varphi$ and \Cref{lemma:LinearDependenceInTopDegree} shows that when $n$ is odd $\langle \alpha \gamma^{p-1} - (p-1)\mu \gamma^{p-1}\rangle \subseteq \ker \varphi.$
    Since $\alpha^{\epsilon} \mu^\delta \gamma^d$ where $\epsilon, \delta \in \{0,1\}$ and $d \in \mathbb{N}$ is a basis for $\Q\{\alpha,\mu,\gamma\}$ and $\varphi(\alpha^{\epsilon} \mu^\delta \gamma^d) = a^{\epsilon} g^\delta g^d$ it follows from \Cref{prop:LinearBasisC} that $\ker \varphi$ is no larger.
\end{proof}

Having constructed an explicit isomorphism $\varphi$ between $\Q\{\alpha,\mu,\gamma\}/I_n$ and $\OS_n^{\symm_n}$, we have proven the desired presentation contained in \Cref{thm:OSnPresentation}.

\section*{Acknowledgments}

The author is grateful to Vic Reiner, Sheila Sundaram, Anh Hoang Nam Tran, and Craig Westerland for helpful conversations regarding this work.
The author is especially grateful to Reiner for sharing this question with him and to Sundaram for simplifying the proof of the Hilbert series in \Cref{prop:OSinvariantHilb}.
The author is also grateful to Galen Dorpalen--Barry for helpful feedback on an earlier draft of this paper.
The author was partially supported by NSF grant DMS-2053288.
This work was also partially completed while in residence at the ICERM special semester on Categorification and Computation in Algebraic Combinatorics.
\bibliographystyle{siam}
\bibliography{main}

\end{document}